\documentclass{article}
\usepackage{amsmath,amsfonts,amssymb,amstext,amsthm,amscd,graphics}

\usepackage[utf8]{inputenc}

\def\C{\mathbb{C}}

\def\P{\mathbb{P}}
\def\B{\mathbb{B}}
{\large }

\def\hfl#1#2{\smash{\mathop{\hbox to 12 mm{\rightarrowfill}}
\limits^{\scriptstyle #1}_{\scriptstyle #2}}}

\def\hflp#1#2{\smash{\mathop{\hbox to 8 mm{\rightarrowfill}}
\limits^{\scriptstyle #1}_{\scriptstyle #2}}}

\newtheorem{theorem}{Theorem}[section]
\newtheorem{definition}[theorem]{Definition}
\newtheorem{corollary}[theorem]{Corollary}
\newtheorem{proposition}[theorem]{Proposition}
\newtheorem{lemma}[theorem]{Lemma}
\newtheorem{remark}[theorem]{Remark}
\newtheorem{example}[theorem]{Example}

\title{Equisingularity in one parameter families of generically reduced curves}
\author{J. Fern\'andez de Bobadilla\footnote{Instituto de Ciencias Matem\'aticas, Madrid, Spain, e-mail: javierbobadilla73@gmail.com} , J. Snoussi\footnote{Corresponding author, Universidad Nacional Aut\'onoma de M\'exico, Instituto de Matem\'aticas, Unidad Cuernavaca, Av. Universidad s/n, Lomas de Chamilpa, 62210, Cuernavaca, Morelos, Mexico, e-mail: jsnoussi@im.unam.mx}   and M. Spivakovsky\footnote{CNRS- Institut de Math\'ematiques de Toulouse and ITMO University, St Petersburg, Russia, e-mail: mark.spivakovsky@math.univ-toulouse.fr} }

\begin{document}

\maketitle

\begin{abstract}
We explore some equisingularity criteria in one parameter families of generically reduced curves. We prove the equivalence between Whitney regularity and Zariski's discriminant criterion. We prove that topological triviality implies smoothness of the normalized surface. Examples are given to show that Witney regularity and equisaturation are not stable under the blow-up of the singular locus nor under the Nash modification.
\end{abstract}

\section{Introduction}
Equisingularity has been first studied for hypersurfaces by O. Zariski. Different concepts can be used to describe the idea of all the fibers in a family of varieties having similar singularities. Many of these concepts consist in measuring some numerical invariant and requiring its constancy along the subspace. However, apart from the case of families of reduced plane curves, where almost all the natural concepts of equisingularity coincide, the general situation is quite ambiguous and unclear.

Even in the case of one parameter family of non planar reduced curves, the classical concepts split into different levels of strength. However, in this case the situation is more or less clear and understood. We refer to \cite{MGranger-equisingularite} for a description of the case of families of reduced curves. We also recommend the expository paper by J. Lipman
\cite{JLipman-equisingularity-obergurgl}, for a general overview of equisingularity problems.

In this work we focus on the case of families of generically reduced curves. These are families where the special fiber may have an embedded component. This happens always when the surface described by the family of curves is not Cohen-Macaulay. When one works with families of curves that are not complete intersections, it is natural to deal with these kind of situations. The examples we give in this work show how natural it is to meet such non Cohen-Macaulay surfaces.

When the family of curves is no longer reduced, many of the ingredients of the classical proofs, and even some definitions, fail. The Milnor number, somehow the star invariant goes to infinity.

Our main purpose in this work consists in comparing some of those equisingularity criteria that can be expressed without requiring the fibers to be reduced. More precisely we consider the Whitney regularity, Zariski's discriminant criterion and topological triviality. We state criteria that are close to normalization in a family and equisaturation.

Our main results are the equivalence between Zariski's discriminant criterion and Whitney regularity together with the fact that topological triviality implies smoothness of the normalized surface. Both results were already known in the case of reduced curves. Specialists expected them to hold also in our case, but no written proof of it is known to the authors. As a tool in one of the proofs, we establish a natural lemma which can be viewed as a partial analogue of the Rolle theorem in the one dimensional complex case.

All along this work we give examples. Some of them illustrate the concepts which appear in the text and others are counterexamples to seemingly natural conjectures. The reader will find an example of a surface Whitney regular along its singular locus that is not Cohen Macaulay and examples that illustrate how some equisingularity criteria are stronger than others.

We also explore the conservation of Whitney regularity and Zariski's discriminant criterion after two types of modifications: the blow up of the singular locus and the Nash modification. We give examples showing that neither of these properties is stable under these modifications.

\section{Equisingularity on surfaces}
Let $(S,0)$ be a germ of a reduced and irreducible complex surface with one-dimensional smooth singular locus.
A generic projection $\pi : (S,0) \rightarrow (\C,0)$ exhibits the surface $S$ as a one-parameter flat
deformation of the curve $X_0 : = \pi^{-1}(0)$.

For a generic projection $\pi$ the curve $X_0$ is reduced at its generic point. Hence, when $S$ is Cohen-Macaulay at the origin, such a projection makes $S$ into a flat deformation of a reduced curve.

In the paper \cite{MGranger-equisingularite}, the authors give a full description of different types of
equisingularity conditions and explain the relations between them, especially in the case of deformations of
reduced curves. Many of these criteria are given in terms of invariants such as Milnor number,
multiplicity, Milnor number of a generic planar projection, dimensions of various ``tangent cones'', and also
Whitney regularity conditions, equisaturation, Zariski discriminant criterion, and simultaneous resolution.

Since we do not want to restrict to the Cohen-Macaulay case, we will not have deformations of reduced curves, therefore we will not be able to use the Milnor number invariant. We will mainly focus on two equisingularity criteria:
Whitney regularity conditions and Zariski's discriminant criterion. However we will also briefly consider topological triviality, normalization in a family and equisaturation or strong equisingularity.

\begin{definition}\label{whitney}
Let $(S,0)$ be a germ of a complex surface with one dimensional smooth singular locus ${\mathcal C}$.
We will say that $(S,0)$ is \textbf{Whitney regular} if a small representative $(S\setminus {\mathcal C}, {\mathcal C})$ satisfies Whitney conditions $(a)$ and $(b)$ at 0, {\it that is}:

For any sequences of points $(x_n)\subset S\setminus {\mathcal C}$ and $(y_n)\subset {\mathcal C}\setminus
\{0\}$ both converging to $0$ and such that the sequence of lines $(x_ny_n)$ converges to a line $l$ and the sequence of directions of tangent spaces, $T_{x_n}S$, to $S$ at $x_n$,
converges to a linear space $T$ we have:

$a)$ the direction of the tangent space $T_0{\mathcal C}$ to ${\mathcal C}$ at the origin is such that
$T_0{\mathcal C}\subset T$

and

$b)$ $l\subset T$.
\end{definition}

It is not hard to prove that condition $(b)$ implies condition $(a)$, see for example \cite{DTrotman-ENS}.

\begin{definition}
Let $(S,0)$ be a germ of a complex surface and $\pi : S \rightarrow \C^2$  a finite map.
The critical locus of $\pi$ is the set of points in $S$ where the map $\pi$ does not
induce a local isomorphism. The discriminant locus of $\pi$ is the set-theoretic image of the critical locus of $\pi$.
\end{definition}

Notice that the singular locus of a surface $S$ is always a subset of the critical
locus of any finite map to $\C^2$.

Zariski's discriminant criterion in dimension two can be stated as follows:

\begin{definition}
Let $(S,0)$ be a germ of a complex surface with one-dimensional smooth singular locus ${\mathcal C}$.
We will say that $(S,0)$ satisfies \textbf{Zariski's discriminant criterion} if any finite map induced by a generic
linear projection from the ambient space to $\C^2$ has smooth discriminant locus.
\end{definition}

The discriminant of a map can be endowed with a scheme structure using Fitting ideals as
in \cite{BTeissier-hunting}. Zariski's discriminant criterion can then be stated in terms
of constancy of the multiplicity of the discriminant.
In \cite[Theorem III. 5]{MGranger-equisingularite} the authors prove that when $(S,0)$ as above is a complete
intersection then Zariski's discriminant criterion is equivalent to Whitney regularity.
Their proof is based on the fact that in this case, the discriminant space is
a hypersurface of $\C^2$ and they apply the L\^e-Greuel formula, relating the Milnor number of a curve
and the multiplicity of the discriminant.

When the surface is not a complete intersection, the discriminant space (with scheme structure)
may have an embedded component. However, if we define the divisorial discriminant to be the closure
of the discriminant space without the origin ({\it i.e.} we ignore the embedded component)
then we still can use L\^e-Greuel formula whenever we have a deformation of a reduced curve
as in \cite[Definition 3.2]{JSnoussi-CMH} and \cite[3.4]{RBondil-DTLe-Trends}.
Therefore, when $(S,0)$ is a Cohen-Macaulay surface, we can still prove the
equivalence between Whitney regularity and Zariski's discriminant criterion
using the same technique as in \cite[Theorem III. 5]{MGranger-equisingularite}.

Before we continue, we would like to give an example showing that a surface as above which is
Whitney regular need not be Cohen Macaulay.
This fact is known to specialists, however we do believe it is useful to make it explicit.

\begin{example}\label{whitneynocm}
\end{example}
Let $(S,0)$ be the surface parametrized by the map:
$$\rho: (a,t) \mapsto (a, t^3, t^4, at^5);$$
it is a family of space curves degenerating to a planar cusp.

The singular set of the surface $(S,0)$ is the curve $(\C\times \{(0,0,0)\}, 0)$.

Consider the linear projection
$$r: \C^4 \rightarrow \C \times \{(0,0,0)\}.$$
In order to prove the condition $(b)$, it is enough to prove the condition $(a)$ and that,
for every sequence of points $(p_n)$ in the non-singular locus of $S$ converging
to the origin, the sequence of lines generated by $p_n$ and $r(p_n)$ converges
to a line contained in the limit $T$ of tangent spaces $T_{p_n}S$.

The Jacobian matrix of $\rho$ is
$$\left(
\begin{array}{cccc}
1 & 0 & 0 & t^5 \\
0 & 3t^2 & 4t^3 & 5at^4\\
\end{array}
\right)
$$

At each non-singular point $p_n$ of $S$ the tangent plane is spanned by the rows of
the Jacobian matrix at $p_n$. At the limit all these planes contain the vector
$(1,0,0,0))$, which spans the tangent line to the singular locus at $0$.
Hence conditon $(a)$ is verified.

Let $l_n$ be the line containing the points $p_n= \rho(a_n,t_n)$ and $r(p_n)$.
The direction of $l_n$ is given by the vector
$(0, t_n^3, t_n^4, a_nt_n^5)$.

After dividing through by $t_n^3$, we see that all these lines, independently of
the choice of $a_n$ and $t_n$ converge to the line $l$ generated by the vector $(0,1,0,0)$.

On the other hand, the tangent space $T_{p_n}$ is spanned by the vectors $(1, 0, 0, t_n^5)$
and $(0, 3t_n^2, 4t_n^3, 5a_nt_n^4)$. Dividing the second vector by $t_n^2$,
we see that all these tangent spaces, independently of the choice of $a_n$ and $t_n$ converge
to the linear space spanned by the vectors $(1,0,0,0)$ and $(0,1,0,0)$ which contains the line $l$.

Thus the  surface $S$ is Whitney regular in the sense of definition \ref{whitney}.

Let us now compute equations of the surface $S$.

If we use the local co-ordinates $(x,y,z,w)$ in the ambient space, the surface $S$
can be defined by the polynomials
$$y^4 - z^3, yw - xz^2, zw - xy^3, x^3y^5 - w^3, x^2y^2z - w^2.$$

It is a reduced surface. The hyperplane section $S \cap (x=0)$ has an embedded component at the origin.
Thus the surface $(S,0)$ is not Cohen-Macaulay at the origin.


\section{Whitney regularity is equivalent to Zariski's discriminant criterion in dimension two}

We will prove now that Whitney regularity conditions are equivalent to Zariski's discriminant
criterion for surfaces. This equivalence has been proved for complete intersections, still in dimension two,
by Brian\c{c}on, Galligo and Granger in \cite[Theorem III. 5]{MGranger-equisingularite}.

\begin{theorem}\label{equisingularidad}
 Let $(S,0)$ be a germ of reduced and irreducible complex surface with a non-singular one-dimensional singular locus ${\mathcal C}$.
The surface is Whitney regular if and only if it satisfies
Zariski's discriminant criterion.
\end{theorem}

In order to prove the theorem we will need to use a criterion for Whitney regularity, established in a general setting by B. Teissier
in \cite[Theorem V.1.2]{BTeissier-rabida}.  It relates Whitney regularity of a pair of strata to the constancy of
the multiplicity of the family of polar varieties along the small stratum.

Let us first recall some definitions.

Consider a germ of analytic surface $(S,0) \subset (\C^N,0)$ and a linear projection $p: \C^N \rightarrow \C^2$.

\begin{definition}
When the restriction of the projection $p$ to a representative $S$ of $(S,0)$ is finite, the closure in $S$ of the critical locus of the restriction of $p$ to the non-singular locus of $S$ is called the polar curve associated to $p$ on $S$ at $0$.

\end{definition}

It is well known that for a generic projection $p$, the associated polar curve is either empty or
one-dimensional. Furthermore, its multiplicity at the origin does not depend on the choice of the
(generic) projection; see for example \cite[Chapter IV]{BTeissier-rabida}.

B. Teissier proved in \cite[Theorem V.1.2]{BTeissier-rabida}, that Whitney regularity is equivalent to equimultiplicity of the polar varieties along the small stratum. In dimension two, there are only two polar varieties: the polar curve and the surface itself. Since a generic projection at a given point is still generic at a nearby point, equimultiplicity of the polar curves along the singular locus of a surface implies that the general polar curve is empty.

Therefore B. Teissier's result, restricted to surfaces, can be restated as follows:

\begin{lemma}\label{teissier}
Let $S$ be a representative of a complex surface germ $(S,0)$ whose singular locus is smooth and one-dimensional.

The surface is Whitney regular if and only if the general polar curve at the origin is empty and the surface is
equimultiple along its singular locus near the origin.
\end{lemma}

\begin{proof} (of Theorem \ref{equisingularidad})

Let $S$ be a small representative of a germ of a reduced and irreducible complex surface $(S,0)$, with one-dimensional smooth singular
locus ${\mathcal C}$.

Assume that the strata $(S\setminus {\mathcal C}, {\mathcal C})$ satisfy Whitney conditions.

Let $\pi: S \rightarrow \C^2$ be a finite map induced by a general linear projection. By lemma \ref{teissier},
the polar curve associated to $\pi$ is empty, so the critical locus of $\pi$ is  the curve ${\mathcal C}$.
Since this curve is smooth and $\pi$ general,  its image, which is the discriminant locus, is again smooth.
Thus $(S,0)$ satisfies Zariski's discriminant criterion.

Conversely, let $\pi: S \rightarrow \C^2$ be a finite map induced by a general linear projection
$p : \C^N \rightarrow \C^2$, for which the discriminant locus is non-singular.

We are going to prove that the polar curve associated to $\pi$ is empty and that the multiplicity of the surface along its singular locus is constant.

The proof will follow from the six steps below:

{\it Step 1:} The normalization of the germ $(S,0)$ is non-singular. 
Indeed, the projection $\pi : S \rightarrow \C^2$ is finite and has a smooth discriminant locus. The composition map $\pi \circ n$ of 
$\pi$ with the normalization $n$ is also finite and its discriminant locus will be either smooth or empty. A normal surface singularity with smooth or empty discriminant (in $\C^2$) is non-singular (see \cite[Theorem 5.2]{Barth}).

{\it Step 2:} Let $0 \in D\subset \C^2$ be a line and $t\in D_t$ a sufficiently close parallel line to it.
We are going to show that in a sufficiently small neighborhood of the origin, the hyperplane section $\pi^{-1}(D_t)$ is connected.

\begin{lemma}\label{connected}
There exists $\epsilon_0 >0$  such that for every $0 < \epsilon \leq \epsilon_0$ the curve $\pi^{-1}(D_t)\cap \B_{\epsilon}$ is connected, for every $t$ with sufficiently small absolute value; where $\B_{\epsilon}$ is the ball of $\C^N$ centered at the origin with radius $\epsilon$.
\end{lemma}

\begin{proof}
The germ of surface $(S,0)$ is assumed to be irreducible. So its normalization is still a germ and moreover, by {\it Step 1}, it is the germ of a smooth surface. So a sufficiently small neighborhood of $n^{-1}(0)$ is isomorphic to an open neighborhood $U$ of the origin in $\C^2$. The composition map: 
$$U \hflp{\cong}{} n^{-1}(S) \hflp{n}{} S \hflp{\pi}{} \C^2$$
is a finite map with discriminant locus contained in a smooth curve of $\C^2$. 

Following \cite[Corollary 5.3]{FArocaJSnoussi-quasi-ordinary}, there exist a system of coordinates $(x,y)$ in $U$ and a natural number $a$ such that, up to isomorphism in the basis,  the composition map is of the form $(x,y) \mapsto (x^a, y)$. This is a consequence of the classification of normal quasi-ordinary singularities; see also \cite[III. 5]{Barth}.

So the inverse image of a line $D_t$ will be a curve in $U$ with an equation of the form $\alpha x^a + \beta y +c = 0$. Such a curve is connected. Its image by the normalization is still a connected curve in $S$. Hence, the inverse image by $\pi$ of $D_t$ is a connected curve in a small representative of $S$.
\end{proof}

{\it Step 3:} A Rolle--type lemma for complex curves.

\begin{lemma}\label{rolle}
Let ${\mathcal T} \subset \C^N$ be an analytic curve. Let ${\B}_{\tau , \epsilon}$ be the open ball of $\C^N$ centered at a point 
$\tau \in {\mathcal T}$ with radius $\epsilon$.
Suppose the intersection ${\mathcal T} \cap {\B}_{\tau, \epsilon}$ is connected.

Consider a linear projection $ \C^N \rightarrow \C$ that induces a finite ramified covering $\rho : {\mathcal T}\cap {\B}_{\tau, \epsilon} \rightarrow \C$, with $\rho (\tau) =0$. 

If $\rho^{-1}(0) \neq \{\tau\}$ then the map $\rho$ has a critical point different from the points in the fiber $\rho^{-1}(0)$.
\end{lemma}

\begin{proof}
This lemma will follow from a Hurwitz formula for possibly singular curves. Let us call: 

$d$ the degree of $\rho$

$n$ the cardinality of $\rho^{-1}(0)$ ; by hypothesis $n\geq 2$

$\chi$ the Euler Characteristic of ${\mathcal T} \cap {\B}_{\tau, \epsilon}$

$\chi_0$ the Euler characteristic of $\C$

Let us suppose that there are no critical points for $\rho$ in ${\B}_{\tau, \epsilon}$ outside $\rho^{-1}(0)$.

We can then chose a triangulation of a ball of $\C$, in such a way that the origin is a vertex and no critical value lies in any edge nor a face. This triangulation is lifted by $\rho$ to a triangulation where all the points of $\rho^{-1}(0)$ are vertices.

So we obtain a Hurwitz formula for this situation
$$\chi  = d \chi_0 + n - d.$$

Knowing that Euler characteristic of $\C$ is 1 we have
$$\chi = n.$$

On the other hand, Euler characteristic is the alternating sum of dimensions of homology spaces.
Since ${\mathcal T}\cap {\B}_{\tau , \epsilon}$ is connected of real dimension two and not compact, we have  
$$\chi = 1 -  h_1({\mathcal T}\cap{\B}_{\tau, \epsilon}) \leq 1,$$
which contradicts the fact that $n\ge2$.

Hence there is necessarily a critical point of $\rho$ outside $\rho^{-1}(0)$. Note that this point may be a singular point of ${\mathcal T}\cap {\B}_{\tau , \epsilon}$.
\end{proof}

{\it Step 4:} The inverse image by $\pi$ of the discriminant locus is the singular locus ${\mathcal C}$ of $S$. 

In fact, consider a point $d$ in the discriminant locus of $\pi$, close to the origin. Consider a general line $D_d\in \C^2$ containing $d$. Since the discriminant locus of $\pi$ is non singular, the intersection of $D_d$ with the discriminant is precisely the point $d$. 

The inverse image $\pi^{-1}(D_d)$ is a curve in $S$.  By Lemma \ref{connected}, it is connected. 

The restriction of $\pi$ to $\pi^{-1}(D_d)$ has all its critical values in the intersection of $D_d$ with the discriminant locus of $\pi$. It has then only one critical value, $d$. By Lemma \ref{rolle} the inverse image $\pi^{-1}(d)$ consists of one point that lies in the singular locus ${\mathcal C}$ of $S$.

{\it Step 5:} The general polar curve is empty. In fact, the critical locus of $\pi$ is in the inverse image of the discriminant locus. By {\it Step 4}, this inverse image is the singular locus of $S$.

{\it Step 6:} The multiplicity of $S$ along ${\mathcal C}$ is constant. In fact, let $c\in {\mathcal C}$ be close to the origin. Since
$\pi$ is generic, the multiplicity of $S$ at $0$ is the degree of $\pi$ at $0$; {\it i.e.} $$m(S,0) = {\rm deg}_0 \pi.$$

In the same way
$$m(S,c) = {\rm deg}_c \pi.$$

The conservation of the degree implies
$${\rm deg}_0 \pi = {\displaystyle \sum_{x\in \pi^{-1}(\pi(c))}}{\rm deg} _x \pi.$$

By {\it Step 4}, $$\pi^{-1}(\pi(c)) = c.$$

So 
$$m(S,0) = {\rm deg} _0 \pi = {\rm deg}_c \pi = m(S,c).$$

This ends the proof of \ref{equisingularidad}
\end{proof} 

\begin{remark}

1) D.T. L\^e and B. Teissier proved in \cite[Theorem 5.3.1]{DTLe-BTeissier-CESPCW2} that Whitney conditions
are equivalent to the constancy
of the Euler characteristic of plane sections of all possible dimensions along the small stratum. This appears implicitly in our proof, since the main tool was an Euler characteristic calculation.

2) In lemma \ref{connected} we used an argument on quasi-ordinary singularities to conclude on the connectedness of the fibers we consider. We could have used a much stronger result by H. Hamm and D.T. L\^e in \cite[Theorem II.1.4]{HHamm-DTLe-creil1988}.
\end{remark}

\section{Topological triviality and smoothness of the normalization}

Let us state two other equisingularity criteria valid for non Cohen-Macaulay surfaces: smoothness of the normalization and topological triviality. The first one is a weaker version of simultaneous normalization in a family. The second one states that the family of curves is homeomorphic to a product of the special fiber with the base. We will compare these criteria to the previous ones. 

\vglue .3cm

{\bf Normalization in a family}

In the case of a flat family of reduced curves, one has normalization in a family when the normalization of the surface induces a normalization on each curve of the family. In particular, it implies that the normalized surface is non-singular.

When the surface is not Cohen-Macaulay, the special fiber has an embedded component at the origin. 
So its total ring of quotients is equal to its ring of holomorphic functions and hence, there is no ``reasonable" notion of normalization.

Instead of normalization in a family we can use the weaker condition of smoothness of the normalized surface. 

From {\it Step 1} in the proof of theorem \ref{equisingularidad} we obtain the following consequence:

\begin{corollary}\label{smoothnormalization}
 Let $(S,0)$ be a germ of a reduced complex analytic surface satisfying Whitney's conditions along its smooth one-dimensional singular locus. The normalization of $(S,0)$ produces a non-singular surface.
 \end{corollary}

It is well known that the converse is not true. See Example \ref{topotrivial} below.

\vglue .3cm

{\bf Topological triviality}

Let $(S,0)$ be a germ of a reduced surface with smooth singular locus ${\mathcal C}$.  Let $r : (S,0) \rightarrow ({\mathcal C}, 0)$ be a retraction making $(S,0)$ into a family of curves. 

\begin{definition}
The family of curves $r: S \rightarrow {\mathcal C}$ is said to be topologically trivial along ${\mathcal C}$  if there exists a homeomorphism $h: S \rightarrow r^{-1}(0) \times {\mathcal C}$ such that $r =\pi \circ h$, where $\pi: r^{-1}(0) \times {\mathcal C} \rightarrow {\mathcal C}$ is the natural projection.
\end{definition}

It is well known that Whitney regularity implies topological triviality (see \cite{JMather-notes}). 

In the surface $y^2 - x^2(x+z)=0$, with projection to the $z$-axis,  the special fiber is a cusp, and the general one has two branches. It shows that one can have smoothness of normalized surface without having topological triviality. 

When we have a family of reduced curves $f: S \rightarrow D$ with a section $\sigma$ such that the fibers $f^{-1}(t)$ are all non singular outside $\sigma(t)$, R.-O. Buchweitz and G.-M. Greuel proved in \cite[Thm.5.2.2]{ROBuchweitz-GMGreuel}, that topological triviality is equivalent to the constancy of the Milnor number of the fibers and equivalent to weak simultaneous resolution. In particular, topological triviality implies smoothness of the normalization. 

As a consequence we obtain:

\begin{lemma}\label{topotriv-normal}
A normal surface singularity is a topologically trivial family of curves if and only if the surface is non singular.
\end{lemma}

We are now going to prove that topological triviality implies smoothness of the normalization, even when the surface is not
Cohen-Macaulay.

\begin{theorem}\label{topotri-smoothness}
Let $(S,0)$ be a germ of a reduced surface with a smooth one dimensional singular locus ${\mathcal C}$.
Let $r: S \rightarrow {\mathcal C}$ be a topologically trivial family of curves. Then the normalization of the surface $(S,0)$ is
non-singular. 
\end{theorem} 

\begin{proof}
It is enough to prove the statement for each analytically irreducible component of the surface $(S,0)$. So we may assume the surface is reduced and irreducible.

Let $r: S \rightarrow {\mathcal C}$ be the retraction making $S$ into a topologically trivial family of curves. The triviality, together with the irreducibility of $S$ at the origin, imply that the fibers $r^{-1}(t)$ are irreducible and generically reduced for all $t$. The surface $S$ is then analytically irreducible at every point.

Call $n: {\bar S} \rightarrow S$ the normalization of $(S,0)$. The inverse image by $n$ of any point of ${\mathcal C}$ is a single point. Then the normalization $n$ induces a homeomorphism between ${\bar S}$ and $S$. Call $\rho = r\circ n$ the composition map. The projection $\rho : {\bar S} \rightarrow {\mathcal C}$ is a topologically trivial family of curves. 

Since the surface ${\bar S}$ is normal, and hence Cohen-Macaulay, the family of curves $\rho : {\bar S} \rightarrow {\mathcal C}$ is a topologically trivial family of reduced curves. So we can apply lemma \ref{topotriv-normal}. The surface ${\bar S}$ is then non singular.

\end{proof}

\begin{remark}

1) Since Whitney conditions imply topological triviality, Theorem \ref{topotri-smoothness} is a stronger statement than Corollary \ref{smoothnormalization}. In fact, It is well known that Whitney regularity is not equivalent to topological triviality; \cite[Chap. V]{MGranger-equisingularite}.

2) In order to prove lemma \ref{topotriv-normal} we can use Mumford's criterion for smoothness of normal surfaces. Indeed, if $(S,0)$ is a normal surface with a topologically trivial family of reduced curves along a non-singular curve, then a suitable representative $S$ of $(S,0)$ is a product of two discs. Its link will be homeomorphic to a sphere.

\end{remark} 

Let us now give an example of a topologically trivial family of curves for which the surface is not Whitney regular. We do believe it is worth to give such an example, where the special curve is not reduced, and therefore we do not use the constancy of Milnor number. 

\begin{example}\label{topotrivial}
\end{example}
Consider the parametrized surface $S\subset \C^4$ given by
$$n: (a,t) \mapsto (a,t^3, t^5, at^2)$$
and consider local coordinates $x$, $y$, $z$ and $w$ in $\C^4$.

The singular locus of $S$ is the $x$-axis.
One can easily check that the map $n$ is bijective and is an isomorphism outside the $x$-axis. So it is the normalization of $S$. 
The family of curves $(t^3, t^5, at^2)$ parametrized by $a$ is a topologically trivial family.
In fact the map 
$$((0,t^3,t^5),a) \mapsto (a, t^3, t^5, at^5/t^3), {\rm whenever} \, t \neq 0$$
and $$((0,0,0),a) \mapsto (a,0,0,0)$$
is a homeomorphism.

However the surface $S$ does not satisfy Whitney conditions along its singular locus.

In fact, consider the sequence of points $p_k = n(1/k, 1/k)$ and $q_k = (1/k , 0, 0, 0)$.
The line $(p_kq_k)$ is spanned by the vector $(0, 1, 1/k^2, 1)$. so the sequence of lines $(p_kq_k)_k$ converges to the line spanned by $(0, 1, 0, 1)$. 

On the other hand, the tangent space $T_{p_k}S$ is spanned by the vectors $(1, 0, 0, 1/k^2)$ and $(0, 3, 5/k^2, 2)$. The limit is the plane spanned by $(1, 0, 0, 0)$ and $(0, 3, 0, 2)$ which does not contain the vector $(0, 1, 0, 1)$.

\section{Equisaturation}

O. Zariski in \cite{OZariski-equising3}, and F. Pham together with B. Teissier in \cite{FPham-BTeissier-fractions}, introduced concepts of saturation of reduced local analytic algebras. In both cases these saturated algebras are intermediate rings between the ring of holomorphic functions and its integral closure. 

Since in this paper we are interested in the case of curves which may not be reduced, we cannot use this definition of saturation. 

Following \cite{MGranger-equisingularite, FPham-BTeissier-fractions, Stutz-equising}, one knows that a flat deformation of a reduced plane curve is equisaturated if and only if the resulting surface is Whitney regular and the generic plane projection of the fibers has a fixed topological type. 

We can use the latter condition as an equisingularity criterion instead of equisaturation.

More precisely, let $(S,0)$ be a germ of a complex surface singularity with a one-dimensional smooth singular locus. Via a projection to
$\C$ we can view this surface as a one-parameter deformation of a curve. We say that $(S,0)$ is strongly equisingular along its singular locus near the origin if it is Whitney regular and the topological type of a generic planar projection of the curves is constant; see \cite{MGranger-equisingularite}.

\begin{example}\label{whitneynosaturado}
\end{example}
Consider the surface given by the parametrization:
$$(a,t) \mapsto (a, t^4, at^6, t^7)$$

It is a non-Cohen-Macaulay surface with one-dimensional smooth singular locus.
One can see that the limit of tangent spaces at the origin is unique; it is given by the linear space spanned by the vectors $(1,0,0,0)$ and $(0,1,0,0)$. For sequences of points $p_n =(a_n, 0, 0, 0)\in {\rm Sing}(S)$ and $q_n = (a_n, t_n^4, a_nt_n^6, t_n^7)\in S\setminus {\rm Sing}(S)$, the corresponding secants $l_n= (p_nq_n)$ converge to the line $l$ spanned by $(0,1, 0, 0)$. So the limit of tangent spaces contain the tangent line to the singular locus and the limit $l$. The surface is then Whitney regular.

However a generic plane projection will have the following parametrization
$$x = t^4, \ y= at^6 + t^7$$

The characteristic exponents change when $a$ takes the value $0$. 
This surface is not strongly equisingular.

\section{Equisingularity criteria are not stable under modifications}

One of our early motivations in this work was to investigate equisingularity criteria possibly stable under Nash modification in dimension two.

We obtained two negative examples, showing that neither Whitney regularity nor strong equisingularity is stable under Nash modification or the blow-up of the singular locus; {\it i.e.} the surface obtained by the modification does not satisfy the equsingularity criterion satisfied by the original surface.

Let us first recall the definition of Nash modification and establish some properties in case of equisingular surfaces.

Let $(X,0) \subset (\C^N,0)$ be a germ of reduced equidimensional analytic space of dimension $d$.
Call ${\bf G}(d,N)$ the Grassmannian of $d$-dimensional linear subspaces of $\C^N$.
Let us denote by $\gamma : X \setminus {\rm Sing}(X) \rightarrow {\bf G}(d,N)$ the Gauss map that associates to any non-singular point
$x$ in a representative of $(X,0)$ the direction of the tangent space $T_xX$. 

Call ${\tilde X}$ the closure of the graph of $\gamma$ in $X \times {\bf G}(d,N)$.
The induced map
$$\nu : {\tilde X} \rightarrow X$$
is, by definition, the Nash modification. For its properties, one can read \cite[Chap. II]{BTeissier-rabida}, \cite{Nobile-nash} and for its desingularization properties see \cite{MSpivakovsky-sandwich, GGonzalez-Sprinberg-doubles}, and also \cite{PGonzalez-BTeissier-nash-toric,  DDuarte-nash-toric-preprint}.

Recall that if a subspace $Y\subset X$ is defined by an ideal $I = (f_1, \cdots, f_r)$, then the blow-up of $X$ along $Y$ is obtained as the closure of the graph of the map
$$\begin{array}{rcl}
X \setminus Y & \rightarrow & \P^{r-1}\\
x & \mapsto & (f_1(x) : \cdots : f_r(x)) 
\end{array}
$$

Assume that we have a germ of reduced and irreducible analytic surface $(S,0)$ with one-dimensional smooth singular locus. Furthermore, assume it is Whitney regular in the sense of definition \ref{whitney}.

\begin{proposition}
A surface $(S,0)$ as above admits a bijective parametrization $n: (\C^2, 0) \rightarrow (S,0)$ that factors through the Nash modification and the blow-up of the singular locus. 
\end{proposition}

\begin{proof}
We have seen in \ref{smoothnormalization} that the normalization of such a surface is smooth. 
Since the surface $(S,0)$ is irreducible, the inverse image of the origin by the normalization map is a single point.  
So, up to an analytic isomorphism, the normalization induces a finite map $(\C^2, 0) \rightarrow (S,0)$, that is an isomorphism outside the singular locus.

Since the surface is topologically trivial along its singular locus, it is analytically irreducible at each point of the singular locus. So, the normalization induces a homeomorphism $(\C^2,0) \rightarrow (S,0)$. 

In \cite[V. I.2]{BTeissier-rabida} it is shown that for a Whitney regular surface, both Nash modification and the blow-up of the singular locus are finite maps. So, by universal property of the normalization, the normalization map factors through Nash modification, and through the blow-up of the singular locus.
\end{proof}

\begin{example}\label{whitneyunstable}
\end{example}
We give an example of an irreducible Whitney regular surface whose Nash modified surface and blow-up along the singular locus are no longer Whitney regular.

Consider the parametrized surface of Example \ref{whitneynosaturado}
$$\begin{array}{rcl}
n: \C^2 & \rightarrow & S \\
(a, t) & \mapsto & (a, t^4, at^6, t^7)\\
\end{array}$$
We have already seen that it is a Whitney regular surface.

The blow-up of the singular locus is given by the parametrization
$$\tau: (a, t) \mapsto (a, t^4, at^2, t^3).$$
Consider the sequence of points $p_n = \tau(\frac{1}{n}, 0)$ and $q_n = \tau(\frac{1}{n}, \frac{1}{n})$, with $n$ a positive integer.

The sequence of corresponding lines $(p_nq_n)$, converges to the line spanned by $(0, 0, 1, 1)$.
The tangent spaces $T_{q_n}S$ are spanned by the vectors $(1, 0, \frac{1}{n^2}, 0)$ and $(0, \frac{4}{n^3}, \frac{2}{n^2}, 
\frac{3}{n^2})$. They converge to the linear space spanned by $(1,0, 0, 0)$ and $(0, 0, 2, 3)$; it does not contain the vector
$(0, 0, 1, 1)$. 

The Nash modification with Plucker coordinates is given by the parametrization
$$\sigma: (a, t) \mapsto (a, t^4, \frac{3}{2} at^2, \frac{7}{4} t^3)$$
With the sequence of points $p_n = \sigma(\frac{1}{n}, 0)$ and $q_n = \sigma (\frac{1}{n}, \frac{1}{n})$, we obtain as limit of lines, the one spanned by $(0, 0, 6, 7)$;  and as limit of directions of tangent spaces the plane spanned by $(1, 0, 0, 0, 0)$ and
$(0, 0, 12, 21)$. The limit of lines is not contained in the limits of tangent spaces.

So, neither Nash modification nor the blow-up of the singular locus is Whitney regular.

\begin{remark}
Example \ref{whitneyunstable} shows that one can obtain a finite map from a surface to $\C^2$ whose discriminant locus is smooth, but the surface does not satisfy Zariski discriminant criterion. In fact, the original surface is Whitney regular and hence satisfies  Zariski's discriminant criterion. There exists then a finite generic projection $\pi : S \rightarrow \C^2$ whose discriminant locus is non-singular.  Composing $\pi$ with any of the modifications considered is still a finite map with the same discriminant locus. But it is not Whitney regular. So a non-generic finite map can have a smooth discriminant locus while the generic one has a singular discriminant locus.
\end{remark} 

\begin{example}\label{equisaturationunstable}
\end{example}

We give here a surface that is strongly equisingular, however, the surfaces obtained by Nash modification and the blow-up of the singular locus are no more strongly equisingular. 

Consider the surface $S$ given by the  parametrization:
$$n: (a, t) \mapsto (a, t^5, t^8, at^9).$$

It is not difficult to check that it is a Whitney regular surface. Furthermore, the plane curve parametrized by 
$$x= t^5, \, y= t^8 + a t^9$$ 
has a topological type that does not depend on the parameter $a$. So the surface is strongly equisingular along its singular locus. 

The parametrization/normalization  $n$ factors through the blow-up of the singular locus and through the Nash modification.

The blow-up of the singular locus is given by
$$\tau : (a, t) \mapsto (a, t^3, at^4, t^5)$$

and the Nash modification is given by
$$\sigma : (a,t) \mapsto (a, t^5, \frac{8}{5} t^3, \frac{9}{5} at^4)$$

Now consider $a$ as a parameter. The generic projection of a curve to $\C^2$ has the following parametrization:
$$x = t^3, \, y= at^4 + t^5$$

whose topological type changes depending on whether $a=0$ or $a\neq 0$.

So the modified surfaces are not strongly equisingular. 
 
 
\providecommand{\bysame}{\leavevmode\hbox to3em{\hrulefill}\thinspace}
\providecommand{\MR}{\relax\ifhmode\unskip\space\fi MR }
\providecommand{\MRhref}[2]{%
  \href{http://www.ams.org/mathscinet-getitem?mr=#1}{#2}
}
\providecommand{\href}[2]{#2}

\end{document}